\documentclass[a4paper,12pt]{amsart}

\usepackage[colorlinks,backref]{hyperref}
\usepackage[english]{babel}
\usepackage{amssymb,amsfonts,amsxtra,amsmath}
\usepackage{stmaryrd}
\usepackage[all,cmtip]{xy}
\usepackage{paralist}
\usepackage[shortlabels]{enumitem}
\usepackage{mathtools}
\usepackage[small]{caption}

\theoremstyle{plain}
\newtheorem{thm}{Theorem}[]
\newtheorem{cor}[thm]{Corollary} 
\newtheorem{prp}[thm]{Proposition} 
\newtheorem{lem}[thm]{Lemma} 
\theoremstyle{definition}
\newtheorem{dfn}[thm]{Definition}
\theoremstyle{remark} 
\newtheorem{rmk}[thm]{Remark}
\newtheorem{ntn}[thm]{Notation}

\newcommand{\NN}{\mathbb{N}}
\newcommand{\ZZ}{\mathbb{Z}}

\newcommand{\E}{\mathcal{E}} 
\newcommand{\F}{\mathcal{F}} 
\newcommand{\K}{\mathcal{K}}
\newcommand{\V}{\mathfrak{V}}

\newcommand{\GSI}{\mathfrak{G}}

\newcommand{\ee}{\mathbf{e}}
\newcommand{\one}{\mathbf{1}}
\newcommand{\ti}{\mathbf{t}}
\newcommand{\mm}{\mathfrak{m}}
\newcommand{\nn}{\mathfrak{n}}

\newcommand{\ol}{\overline}

\newcommand{\wh}{\widehat}

\DeclareMathOperator{\dist}{d}
\DeclareMathOperator{\LL}{L}
\DeclareMathOperator{\PP}{P}

\numberwithin{equation}{section}

\begin{document}

\title{Poincar\'e series on good semigroup ideals}

\author[L.~Tozzo]{Laura Tozzo}
\address{L.~Tozzo\\
Department of Mathematics\\
University of Kaiserslautern\\
67663 Kaiserslautern\\
Germany}
\email{\href{mailto:tozzo@mathematik.uni-kl.de}{tozzo@mathematik.uni-kl.de}}




\begin{abstract}
The Poincar\'e series of a ring associated to a plane curve was defined by Campillo, Delgado, and Gusein-Zade.
This series, defined through the value semigroup of the curve, encodes the topological information of the curve.
In this paper we extend the definition of Poincar\'e series to the class of good semigroup ideals, to which value semigroups of curves belong.
Using this definition we generalize a result of Pol: under suitable assumptions, given good semigroup ideals $E$ and $K$, with $K$ canonical, the Poincar\'e series of $K-E$ is symmetric to the Poincar\'e series of $E$.
\end{abstract}


\maketitle

\section{Introduction}

Plane algebroid curves are determined by their value semigroups up to equivalence in the sense of Zariski, as shown by Waldi \cite{Wal72,Wal00}.
Value semigroups are important invariants of curves also with regard to duality properties.
Kunz \cite{Kun70} was the first to show that the Gorensteinness of an analytically irreducible and residually rational local ring corresponds to a symmetry of its numerical value semigroup.
Waldi \cite{Wal72} gave a definition of symmetry for more branches, and showed that plane (hence Gorenstein) curves with two branches have symmetric value semigroups.
Later Delgado \cite{Del87} proved the analogue of Kunz' result for general algebroid curves: they are Gorenstein if and only if their value semigroup is symmetric. 
Campillo, Delgado and Kiyek \cite{CDK94} extended Delgado result to analytically reduced and residually rational local rings with infinite residue field.
D'Anna \cite{D'A97} then used the definition of symmetry given by Delgado to define a canonical semigroup ideal $K_0$, and showed that a fractional ideal $\K$ of $R$ such that $R\subseteq \K\subseteq \ol R$ is canonical if and only if its value semigroup coincides with $K_0$.
Recently Pol \cite{Pol16} studied the value semigroup ideal of the dual of a fractional ideal over Gorenstein algebroid curves.
In \cite{DSG} the author together with Korell and Schulze gave a new definition of a canonical semigroup ideal $K$ (see Definition \ref{28}) and extended D'Anna's and Pol's results to the larger class of admissible rings (see Definition \ref{88}).
Moreover, one of the main results of \cite{DSG} shows that Cohen--Macaulay duality and semigroup duality are compatible under taking values, if the ring is admissible.
An admissible ring is in particular semilocal, and its value semigroup, as first observed by Barucci, D'Anna and Fr\"oberg \cite{BDF00}, satisfies particular axioms which define the class of good semigroups.

In this paper we analyze further the duality properties of good semigroups by showing symmetry properties of their Poincar\'e series.
In \cite{Sta77}, the author showed that Gorenstein semigroup rings have symmetric Hilbert series.
This is also equivalent to the value semigroup associated to the semigroup ring being symmetric.
Adapting the concept of Hilbert series to value semigroups leads to the concept of Poincar\'e series.
A definition of Poincar\'e series for a plane curve singularity was given by Campillo, Delgado and Gusein-Zade in \cite{CDG03}, where they showed that it coincides with the Alexander polynomial, a complete topological invariant of the singularity.
Moyano-Fernandez in \cite{MF15}, using a definition inspired by the above, analyzed the connection between univariate and multivariate Poincar\'e series of curve singularities and later on, together with Tenorio and Torres \cite{MFTT17}, they showed that the Poincar\'e series associated to generalized Weierstrass semigroups can be used to retrieve entirely the semigroup, hence highlighting the potential of Poincar\'e series.
Later Pol \cite[\S 5.2.8 ]{Pol16} considered a symmetry problem on Gorenstein reduced curves. 
She proved that the Poincar\'e series of the Cohen--Macaulay dual of a fractional ideal $\E$ is symmetric to the Poincar\'e series of $\E$, therefore generalizing Stanley's result to fractional ideals of Gorenstein rings.
Pol's result strongly uses the fact that it is always possible to define a filtration on value semigroups (see Definition \ref{01}), as done first in \cite{CDK94}.
To deal with this filtration an important tool is the distance $\dist(E\backslash F)$ between two good semigroup ideals $E\subseteq F$ (see Definition~\ref{06}).
Using the notion of distance and the duality on good semigroups given in \cite{DSG}, we are able to generalize Pol's result to good semigroup ideals.
We prove that, given good semigroup ideals $E$ and $K$, with $K$ canonical, the Poincar\'e series of $K-E$ is symmetric to the Poincar\'e series of $E$ under suitable assumptions.
In particular, the symmetry is true (without additional assumptions) whenever $E$ is the value semigroup of a fractional ideal $\E$ of an admissible ring $R$.

\section{Preliminaries}

In this section we recall definitions and known results that will be needed in the rest of the paper.

Let $S\subseteq \ol S$ be a partially ordered cancellative commutative monoid, where $\ol{S}$ is a partially ordered monoid, isomorphic to $\NN^s$ with its natural partial order.
Then the group of differences $D_S$ of $\ol S$ is isomorphic to $\ZZ^s$.
In the following we always fix an isomorphism $D_S\cong\ZZ^s$, in order to talk about indexes $i\in\{1,\dots,s\}$.

\subsection{Good semigroups and their ideals}

The following where first defined in \cite[\S1]{Del88} and \cite[\S2]{D'A97}.

\begin{dfn}
Let $E\subseteq \ZZ^s$.
We define properties:
\begin{enumerate}[label={(E\arabic*)}, ref=E\arabic*]
\setcounter{enumi}{-1}
\item\label{E0} There exists an $\alpha\in \ZZ^s$ such that $\alpha+\NN^s\subseteq E$.
\item\label{E1} If $\alpha,\beta \in E$, then $\min\{\alpha,\beta\} := (\min\{\alpha_i,\beta_i\})_{i\in I}\in E$.
\item\label{E2} For any $\alpha,\beta\in E$ and $j\in I$ with $\alpha_j=\beta_j$ there exists an $\epsilon\in E$ such that $\epsilon_j>\alpha_j=\beta_j$ and $\epsilon_i\ge\min\{\alpha_i,\beta_i\}$ for all $i\in I\setminus\{j\}$ with equality if $\alpha_i\not=\beta_i$.
\end{enumerate}
\end{dfn}

\begin{dfn}
We call $S$ a \emph{good semigroup} if properties \eqref{E0}, \eqref{E1} and \eqref{E2} hold for $E=S$.

A \emph{semigroup ideal} of a good semigroup $S$ is a subset $\emptyset\ne E \subseteq D_S$ such that $E+S\subseteq E$ and $\alpha+E\subseteq S$ for some $\alpha\in S$. 

If $E$ satisfies \eqref{E1}, we denote by $\mu^E:=\min E$ its \emph{minimum}.

If $E$ satisfies \eqref{E1} and \eqref{E2}, then we call $E$ a \emph{good semigroup ideal} of $S$.
Note that any semigroup ideal of a good semigroup $S$ automatically satisfies \eqref{E0}.

If $E$ and $F$ are semigroup ideals of a good semigroup $S$, we define
\[
E-F:=\{\alpha\in D_S\mid \alpha+F\subseteq E\},
\]
and we call 
\[
C_E:=E-\ol S=\{\alpha\in D_S\mid \alpha+\ol S\subseteq E\}
\]
the \emph{conductor ideal} of $E$. 
If $E$ is a semigroup ideal of $S$ satisfying \eqref{E1}, then we call
$\gamma^E:=\mu^{C_E}$ the \emph{conductor} of $E$.
We abbreviate $\gamma:=\gamma^S$ and $\tau:=\gamma-\one$, where $\one=(1,\dots,1)\in\NN^s$.
\end{dfn}


\begin{ntn}
Let $S$ be a good semigroup.
The set of good semigroup ideals of $S$ is denoted by $\GSI_S$.
\end{ntn}


\begin{rmk}\label{21}
Let $S$ be a good semigroup.
For any $E,F\in \GSI_S$ and $\alpha\in D_S$ the following hold:
\begin{enumerate}[label=(\alph*), ref=\alph*]
\item $\alpha+E\in\GSI_S$.
\item $(\alpha+E)-F=\alpha+(E-F)$ and $E-(\alpha+F)=-\alpha+(E-F)$.
\item $E-S=E$.
\end{enumerate}
\end{rmk}


\begin{dfn}\label{01}
Let $S$ be a good semigroup.
For any $E\in\GSI_S$, we define a decreasing filtration $E^\bullet$ on $E$ by semigroup ideals
\[
E^\alpha:=\{\beta\in E\mid \beta\ge\alpha\}
\]
for any $\alpha\in D_S$.
\end{dfn}


\begin{rmk}
Let $S$ be a good semigroup.
For a semigroup ideal $E\in \GSI_S$ we have $E=E^{\mu^E}$ and, by definition of conductor, $C_E=\gamma^E+\ol S=E^{\gamma^E}$.
\end{rmk}

\subsection{Distance of semigroup ideals}

\begin{dfn}\label{06}
Let $E\subseteq D_S$. 
Elements $\alpha,\beta\in E$ with $\alpha<\beta$ are called \emph{consecutive} in $E$ if $\alpha<\delta<\beta$ implies $\delta\not\in E$ for any $\delta\in D_S$.
For $\alpha,\beta\in E$, a chain 
\begin{equation}\label{20}
\alpha=\alpha^{(0)}<\cdots<\alpha^{(n)}=\beta
\end{equation}
of points $\alpha^{(i)}\in E$ is said to be \emph{saturated of length n} if $\alpha^{(i)}$ and $\alpha^{(i+1)}$ are consecutive in $E$ for all $i \in \{0,\ldots, n-1\}$.
If $E$ satisfies
\begin{enumerate}[label={(E\arabic*)}, ref=E\arabic*]
\setcounter{enumi}{3}
\item \label{E4} For fixed $\alpha,\beta\in E$, any two saturated chains \eqref{20} in $E$ have the same length $n$.
\end{enumerate}
then we call $\dist_E(\alpha,\beta):=n$ the \emph{distance} of $\alpha$ and $\beta$ in $E$.
\end{dfn}

Due to \cite[Proposition~2.3]{D'A97}, any $E \in \GSI_S$ satisfies property \eqref{E4}.


\begin{dfn}
Let $S$ be a good semigroup, and let $E \subseteq F$ be two semigroup ideals of $S$ satisfying property \eqref{E4}.
Then we call
\[
\dist(F\backslash E):=d_F(\mu^F,\gamma^E)-\dist_E(\mu^E,\gamma^E)
\]
the \emph{distance} between $E$ and $F$.
\end{dfn}


The following was proved in \cite[Proposition~2.7]{D'A97}:

\begin{lem}\label{07}
If $E\subseteq F\subseteq G$ are semigroup ideals of a good semigroup $S$ satisfying property \eqref{E4}, then
\[
\dist(G\backslash E)=\dist(G\backslash F)+\dist(F\backslash E).
\]
\end{lem}


Moreover, as proved by the author in \cite[Proposition 4.2.6]{DSG}, distance can be used to check equality:

\begin{prp}\label{18}
Let $S$ be a good semigroup, and let $E,F\in\GSI_S$ with $E\subseteq F$.
Then $E=F$ if and only if $\dist(F\backslash E)=0$.
\end{prp}

\subsection{Canonical semigroup ideals}

The following definition is \cite[Definition 5.2.3]{DSG}:

\begin{dfn}\label{28}
Let $S$ be a good semigroup.
A \emph{canonical ideal} $K$ is a good semigroup ideal of $S$ such that $K\subseteq E$ implies $K=E$ for any $E$ with $\gamma^k=\gamma^E$.
\end{dfn}


\begin{ntn}\label{12}
Let $\alpha\in D_S$, $E\subseteq D_S$. 
\begin{itemize}
\item $\Delta^E_i(\alpha)=\{\beta\in E\mid \beta_i=\alpha_i \text{ and } \beta_j>\alpha_j \text{ for all } j\ne i\}$;
\item  $\ol\Delta^E_i(\alpha)=\{\beta\in E\mid \beta_i=\alpha_i \text{ and } \beta_j\ge\alpha_j \text{ for all } j\ne i\}$;
\item $\Delta^E(\alpha)=\cup_{i\in\{1,\dots,s\}}\Delta^E_i(\alpha)$;
\item $\ol\Delta^E(\alpha)=\cup_{i\in\{1,\dots,s\}}\Delta^E_i(\alpha)$.
\end{itemize}
We denote by $\ee_i$ the $i$-th vector of the canonical basis of $D_S$.
Then $\ol\Delta^E_i(\alpha)=\Delta^E_i(\alpha+\ee_i-\textbf{1})$.
\end{ntn}


Using \cite[Proposition 5.2.10]{DSG} and \cite[Proposition 3.2]{D'A97} yields:
 
\begin{prp}\label{30}
Let $S$ be a good semigroup.
Then $K$ is a canonical ideal if and only if $K=\alpha+K_0$ for some $\alpha\in D_S$, where
\[
K_0=\{\alpha\in D_S\mid \Delta^S(\tau-\alpha)=\emptyset\}
\]
is a good semigroup ideal of $S$ called \emph{normalized canonical ideal} of $S$.
In particular, $K_0$ is the only canonical semigroup ideal with $\gamma^{K_0}=\gamma$.
\end{prp}


\begin{lem}\label{14}
Let $S$ be a good semigroup.
If $E\in \GSI_S$, then 
\begin{enumerate}[label=(\alph*), ref=\alph*]
\item\label{14a} $K_0-E = \{\alpha \in D_S \mid \Delta^E(\tau-\alpha)=\emptyset\}\in \GSI_S$;
\item\label{14b} $\gamma^{K_0-E}=\gamma-\mu^E$;
\item\label{14c} $\mu^{K_0-E}=\gamma-\gamma^E$.
\end{enumerate}
\end{lem}
\begin{proof}
For part (a) see \cite[Computation 3.3]{D'A97} and \cite[Lemma 5.2.9.(b)]{DSG}. 
Part (b) is proven in \cite[Lem. 4.1.13]{DSG}.
Part (c) follows by \cite[Theorem 5.2.7(iii)]{DSG}. 
In fact, $\mu^{K_0-E}=\gamma-\gamma^{K_0-(K_0-E)}=\gamma-\gamma^E$.
\end{proof}


In the following, when we talk about \emph{the} canonical semigroup ideal, we refer to $K_0$.
To make notation easier, we will write $K$ instead of $K_0$.
Notice that by Remark \ref{21} and Proposition \ref{30} all the results hold as well for any $K$ canonical, up to translation by a suitable $\alpha$.


\begin{rmk}\label{24}
Let $S$ be a good semigroup, and $E\in\GSI_S$.
For all $\alpha\in D_S$ we have $E-D_S^\alpha=D_S^{\gamma-\alpha}$. 
In fact, Remark \ref{21} implies:
\[
E-D_S^{\alpha}=E-(\alpha+D_S)=-\alpha+(E-D_S)=-\alpha+\gamma+D_S=D_S^{\gamma-\alpha}.
\]
This is in particular true for $E=K$.
\end{rmk}


The following is \cite[Theorem 5.2.6]{DSG}:
\begin{prp}\label{22}
Let $S$ be a good semigroup, $E\in \GSI_S$, and let $K$ be the canonical semigroup ideal. 
Then $K-(K-E)=E$.
\end{prp}

\subsection{Value semigroups}

We now give a few definitions regarding rings, in order to make clear the connection between their value semigroups and good semigroups.

In the following, $R$ is a commutative ring with 1, and $Q_R$ its total ring of fractions.
We always assume fractional ideals of $R$ to be regular, i.e.\ to contain at least a regular element.


\begin{dfn}
A \emph{valuation ring} of $Q_R$ is a subring $V \subsetneq Q_R$ such that the set $Q_R \setminus V$ is multiplicatively closed.

If $R\subseteq V$, we call $V$ a \emph{valuation ring over $R$}.
We denote by $\V_R$ the set of all valuation rings of $Q_R$ over $R$.

A valuation ring $V$ of $Q$ with unique regular maximal ideal $\mm_V$ is called a \emph{discrete valuation ring} if $\mm_V$ is the only regular prime ideal of $V$ (see \cite[Ch.~I, (2.16) Def.]{KV04}).

A \emph{discrete valuation} of $Q_R$ is a map $\nu\colon Q_R\twoheadrightarrow\ZZ\cup\{\infty\}$ satisfying
\begin{equation*}
\nu (xy)=\nu(x)+\nu(y),\quad
\nu(x+y)\geq\min\{\nu(x),\nu(y)\}
\end{equation*}
for any $x,y \in Q_R$.
We refer to $\nu(x)\in\ZZ\cup\{\infty\}$ as the \emph{value} of $x\in Q_R$ with respect to $\nu$.
\end{dfn}


The following theorem is \cite[Ch.~II, (2.11) Thm.]{KV04}, and characterizes valuation rings over one-dimensional semilocal Cohen--Macaulay rings.

\begin{thm}
Let $R$ be a one-dimensional semilocal Cohen--Macaulay ring.
The set $\V_R$ is finite and non-empty, and it contains discrete valuation rings only.
\end{thm}


Thanks to this theorem, we can give the following definition:

\begin{dfn}
Let $R$ be a one-dimensional semilocal Cohen--Macaulay ring, and let $\V_R$ be the set of (discrete) valuation rings of $Q_R$ over $R$ with valuations
\[
\nu=\nu_R:=(\nu_V)_{V\in\V_R}\colon Q_R \to (\ZZ\cup\{\infty\})^{\V_R}.
\]
To each fractional ideal $\E$ of $R$ we associate its \emph{value semigroup ideal}
\[
	\Gamma_\E := \nu(\{x\in \E\mid x \text{ is regular}\}) \subseteq \ZZ^{\V_R}.
\]
If $\E=R$, then the monoid $\Gamma_R$ is called the \emph{value semigroup} of $R$.
\end{dfn}


The following additional definitions are needed to make the value semigroup of a ring into a good semigroup.

\begin{dfn}\label{88}
Let $R$ be a one-dimensional semilocal Cohen--Macaulay ring.
Let $\wh R$ denote its completion at the Jacobson radical and $\ol R$ its integral closure in its total ring of fractions $Q_R$.
\begin{enumerate}[label=(\alph*), ref=\alph*]
\item $R$ is \emph{analytically reduced} if $\wh R$ is reduced or, equivalently, $\wh{R_\mm}$ is reduced for all maximal ideals $\mm$ of $R$. 
\item $R$ is \emph{residually rational} if $\ol{R}/\nn=R/\nn\cap R$ for all maximal ideals $\nn$ of $\ol{R}$.
\item $R$ has \emph{large residue fields} if $\lvert R/\mm \rvert \ge \lvert \V_{R_\mm} \rvert$ for all maximal ideals $\mm$ of $R$.
\item $R$ is \emph{admissible} if it is analytically reduced and residually rational with large residue fields.
\end{enumerate}
\end{dfn}


The following was proven in \cite[Cor.~3.2.3]{DSG}.

\begin{prp}
If $R$ is admissible, then its value semigroup $\Gamma_R$ is a good semigroup, and $\Gamma_\E$ is a good semigroup ideal for any fractional ideal $\E$ of $R$.
\end{prp}


\begin{ntn}
Let $R$ be an admissible ring, and let $\E$ be a fractional ideal of $R$.
For any $\alpha\in D_S$ denote
\[
\E^\alpha:=\{x\in \E\mid \nu(x)\ge\alpha\}.
\]
\end{ntn}


There is a clear link between filtrations of fractional ideals and filtrations of good semigroup ideals (see \cite[Lemma~3.1.3]{DSG}):
\begin{lem}\label{15}
Let $R$ be an admissible ring, and let $\E$ be a fractional ideal of $R$.
Then $\E^\alpha$ is a (regular) fractional ideal of $R$ and $(\Gamma_{\E})^\alpha=\Gamma_{\E^\alpha}$ for all $\alpha\in D_S$. 
\end{lem}


The following was proven first by D'Anna \cite[Proposition~2.2]{D'A97} and then extended in \cite[Proposition 4.2.7]{DSG}.

\begin{prp}\label{17}
Let $R$ be an admissible ring, and let $\E, \F$ be two fractional ideals of $R$ with $\E\subseteq \F$. 
Then
\[
\ell_R(\F/\E)=\dist(\Gamma_\F\backslash \Gamma_\E),
\]
where $\ell_R(\F/\E)$ denotes the length of the quotient $\F/\E$ as $R$-module.
\end{prp}


Finally, \cite[Theorem 5.3.4]{DSG} shows that Cohen-Macaulay duality translates to semigroup duality:

\begin{prp}\label{74}
Let $R$ be an admissible ring with canonical ideal $\K$.
Then
\begin{enumerate}[label=(\alph*), ref=\alph*]
\item\label{74a} $\Gamma_{\K:\F}=\Gamma_\K-\Gamma_\F$ for any fractional ideal $\F$ and
\item\label{74b} $\dist(\Gamma_\K-\Gamma_\E\backslash\Gamma_\K-\Gamma_\F)=\dist(\Gamma_\F\backslash \Gamma_\E)$ for any fractional ideals $\E,\F$ with $\E\subseteq\F$.
\end{enumerate}
\end{prp}

\section{Distance and duality}

We now prove some technical results used in the coming section.

\begin{lem}\label{16}
Let $S$ be a good semigroup, $E\in\GSI_S$, and $\alpha\in D_S$. 
Then $\dist(E^\alpha\backslash E^{\alpha+\ee_i}) \le 1$.
\end{lem}

\begin{proof}
We have the following:
\begin{gather}\label{19}
\begin{aligned}
\dist(E^\alpha\backslash E^{\alpha+\ee_i}) &= \dist_{E^\alpha}(\mu^{E^\alpha},\gamma^{E^{\alpha+\ee_i}})- \dist_{E^{\alpha+\ee_i}}(\mu^{E^{\alpha+\ee_i}},\gamma^{E^{\alpha+\ee_i}}) \\
&=\dist_{E^\alpha}(\mu^{E^\alpha},\gamma^{E^{\alpha+\ee_i}})-\dist_{E^\alpha}(\mu^{E^{\alpha+\ee_i}},\gamma^{E^{\alpha+\ee_i}})
\end{aligned}
\end{gather}
where the first equality is the definition of distance, and the second equality holds because a saturated chain between $\mu^{E^{\alpha+\ee_i}}$ and $\gamma^{E^{\alpha+\ee_i}}$ in $E^{\alpha+\ee_i}$ is also saturated in $E^\alpha$.
Now observe that $\mu^{E^\alpha}$ and $\mu^{E^{\alpha+\ee_i}}$ are always comparable. 
In fact, by minimality of $\mu^{E^\alpha}$ it has to be $\mu^{E^\alpha}=\min\{\mu^{E^\alpha},\mu^{E^{\alpha+\ee_i}}\}\le \mu^{E^{\alpha+\ee_i}}$.
So \eqref{19} becomes
\[
\dist(E^\alpha\backslash E^{\alpha+\ee_i})=\dist_{E^\alpha}(\mu^{E^\alpha},\mu^{E^{\alpha+\ee_i}}).
\]
Now let $\mu^{E^\alpha}=\alpha^{(0)}<\cdots<\alpha^{(m)}=\mu^{E^{\alpha+\ee_i}}$ be a saturated chain in $E$. 
Suppose $m\ge 2$. 
By minimality of $\mu^{E^{\alpha+\ee_i}}$, we have that $\alpha^{(k)}\in \ol\Delta^E_i(\alpha)$ for all $k<m$.
Consider $\alpha^{(0)},\alpha^{(1)}\in E$. 
They have $\alpha^{(0)}_i=\alpha^{(1)}_i=\alpha_i$ and there exists a $j\ne i$ such that $\alpha^{(0)}_j<\alpha^{(1)}_j\le\alpha^{(m)}_j=\mu^{E^{\alpha+\ee_i}}_j$. 
We can apply property \eqref{E2} to $\alpha^{(0)},\alpha^{(1)}\in E$ and obtain a $\beta\in E$ with $\beta_i>\alpha_i$ and
$\beta_j=\min\{\alpha^{(0)}_j,\alpha^{(1)}_j\}=\alpha^{(0)}_j$. 
In particular, $\beta\in E^{\alpha+\ee_i}$.
Thus, by minimality of $\mu^{E^{\alpha+\ee_i}}$, it has to be $\min\{\beta,\mu^{E^{\alpha+\ee_i}}\}=\mu^{E^{\alpha+\ee_i}}$. 
Then
$\mu^{E^{\alpha+\ee_i}}_j=\min\{\beta_j,\mu^{E^{\alpha+\ee_i}}_j\}=\min\{\alpha^{(0)}_j,\mu^{E^{\alpha+\ee_i}}_j\}=\alpha^{(0)}_j<\mu^{E^{\alpha+\ee_i}}_j$. 
This is a contradiction.
Hence the claim.
\end{proof}


\begin{lem}\label{02}
Let $S$ be a good semigroup, and let $E\in\GSI_S$. 
Then $\dist(E^\alpha\backslash E^{\alpha+\ee_i})=1$ if and only if $\ol\Delta^E_i(\alpha)\ne \emptyset$.
\end{lem}
\begin{proof}
Observe that by definition $E^{\alpha}=E^{\alpha+\ee_i}\cup\ol \Delta^E_i(\alpha)$ and $E^{\alpha+\ee_i}\cap\ol \Delta^E_i(\alpha)=\emptyset$.
By Proposition \ref{18}, $\dist(E^\alpha\backslash E^{\alpha+\ee_i})=0$ if and only if  $E^\alpha=E^{\alpha+\ee_i}$, i.e.~if and only if $\ol \Delta^E_i(\alpha)=\emptyset$.
So the claim follows by Lemma \ref{16}. 
\end{proof}


The following proposition characterizes the distance in terms of $\ol \Delta$-sets.

\begin{prp}\label{03}
Let $S$ be a good semigroup, $E\in \GSI_S$, and $\alpha,\beta \in D_S$ with $\alpha\le\beta$.
Then $E^\beta\subseteq E^\alpha$. 

Let $\alpha=\alpha^{(0)}<\alpha^{(1)}<\cdots<\alpha^{(n)}=
\beta$ be a saturated chain in $D_S$, with $\alpha^{(j+1)}=\alpha^{(j)}+\ee_{i(j)}$ for any $j\in\{0,\dots,n-1\}$. 
We have:
\[
\dist(E^\alpha\backslash E^\beta)=|\{j\in\{0,\dots,n-1\}\mid \ol\Delta^E_{i(j)}(\alpha^{(j)})\ne\emptyset\}|,
\]
where $|-|$ denotes the cardinality.
\end{prp}
\begin{proof}
Using the additivity of the distance (see Lemma \ref{07}), our assumptions and Lemma \ref{02} we get the following equalities:
\begin{align*}
\dist(E^\alpha\backslash E^\beta)=& \sum_{j=0}^{n-1}\dist(E^{\alpha^{(j)}}\backslash E^{\alpha^{(j+1)}})
=\sum_{j=0}^{n-1}\dist(E^{\alpha^{(j)}}\backslash E^{\alpha^{(j)}+\ee_{i(j)}})\\
=&|\{j\in\{0,\dots,n-1\}\mid \ol\Delta^E_{i(j)}(\alpha^{(j)})\ne\emptyset\}|.\qedhere
\end{align*}
\end{proof}


As a corollary, we obtain a way to compute the distance between two semigroup ideals.

\begin{cor}\label{23}
Let $S$ be a good semigroup. 
Let $E\subseteq F\in \GSI_S$, and let $\mu^F=\alpha^{(0)}<\alpha^{(1)}<\cdots<\alpha^{(m)}=
\mu^E<\cdots<\alpha^{(n)}=\gamma^E$ be a saturated chain in $D_S$. 
In particular, $\alpha^{(j+1)}=\alpha^{(j)}+\ee_{i(j)}$ for any $j\in\{0,\dots,n-1\}$. 
Then
\begin{align*}
\dist(F\backslash E)=& |\{j\in\{0,\dots,n-1\}\mid \ol\Delta^F_{i(j)}(\alpha^{(j)})\ne\emptyset\}| \\
			&-|\{j\in\{m,\dots,n-1\}\mid \ol\Delta^E_{i(j)}(\alpha^{(j)})\ne\emptyset\}|
\end{align*}
\end{cor}

\begin{proof}
By additivity of the distance (see Lemma \ref{07}) we have:
\[
\dist(F\backslash E)=\dist(F\backslash C_E)-\dist(E\backslash C_E)= \dist(F^{\mu^F}\backslash F^{\gamma^E})-\dist(E^{\mu^E}\backslash E^{\gamma^E}).
\]
The claim follows by Proposition \ref{03}.
\end{proof}


The following two lemmas are necessary to prove Proposition \ref{08}.

\begin{lem}\label{04}
Let $S$ be a good semigroup, and let $E\in\GSI_S$. 
Let $K$ be the canonical ideal of $S$.
If $\ol\Delta^{K-E}_i(\tau-\alpha)\ne\emptyset$ then $\Delta^E_i(\alpha)=\emptyset$.
\end{lem} 

\begin{proof}
Let $\tau-\beta\in\ol\Delta^{K-E}_i(\tau-\alpha)$. Then
\begin{align*}
\tau_i-\beta_i=&\tau_i-\alpha_i,\\
\tau_j-\beta_j\ge&\tau_j-\alpha_j \text{ for all } j\ne i,
\end{align*}
and $\Delta^E(\beta)=\emptyset$ by Lemma \ref{14}.\eqref{14a}.
As $\beta_i=\alpha_i$ and $\beta_j\le \alpha_j$, it follows $\Delta^E_i(\alpha)\subseteq \Delta^E_i(\beta)=\emptyset$.
\end{proof}


\begin{lem}\label{05}
Let $S$ be a good semigroup, $E\in \GSI_S$, and $\alpha,\beta\in D_S$ with $\alpha\le\beta$.
Let $K$ be the canonical ideal of $S$.
Then:
\[
\dist(E^\alpha\backslash E^\beta)\le \dist(D_S^\alpha\setminus D_S^\beta)-\dist((K-E)^{\gamma-\beta}\setminus(K-E)^{\gamma-\alpha}).
\]
\end{lem}

\begin{proof}
Let 
\[
\alpha=\alpha^{(0)}<\alpha^{(1)}<\cdots<\alpha^{(n)}=
\beta
\]
be a saturated chain in $D_S$, with $\alpha^{(j+1)}=\alpha^{(j)}+\ee_{i(j)}$ for any $j\in\{0,\dots,n-1\}$.
Let us denote $J=\{0,\dots,n-1\}$.

Set $\beta^{(j)}=\gamma-\alpha^{(n-j)}$. 
Then
\[
\gamma-\beta=\beta^{(0)}<\beta^{(1)}<\cdots<\beta^{(n)}=\gamma-\alpha
\]
is a saturated chain in $D_S$, and 
\[
\beta^{(j+1)}=\gamma-\alpha^{(n-(j+1))}=\gamma-(\alpha^{(n-j))}-\ee_{i(n-(j+1))})=\beta^{(j)}+\ee_{i(n-(j+1))}.
\]
By Proposition \ref{03} we have $\dist(E^\alpha\backslash E^\beta)=|\{j\in J\mid \ol\Delta^E_{i(j)}(\alpha^{(j)})\ne\emptyset\}|$. 
Recall that $E=K-(K-E)$ by Proposition \ref{22}. 
Therefore we can apply Lemma \ref{04} and obtain
\begin{gather}\label{25}
\begin{aligned}
\dist(E^\alpha\backslash E^\beta) &=|\{j\in J\mid \ol\Delta^E_{i(j)}(\alpha^{(j)})\ne\emptyset\}| \\
&\le |\{j\in J\mid \Delta^{K-E}_{i(j)}(\tau-\alpha^{(j)})=\emptyset\}| \\
&=|\{j\in J\mid \Delta^{K-E}_{i(j)}(\gamma-\alpha^{(j)}-\one)=\emptyset\}| \\
&=|\{j\in J\mid \Delta^{K-E}_{i(j)}(\beta^{(n-j)}-\one)=\emptyset\}| \\
&=|\{j\in J\mid \ol\Delta^{K-E}_{i(j)}(\beta^{(n-(j+1))})=\emptyset\}| \\
&=n-|\{j\in J\mid \ol\Delta^{K-E}_{i(j)}(\beta^{(n-(j+1))})\ne\emptyset\}| \\
&=n-|\{j\in J\mid \ol\Delta^{K-E}_{i(n-(j+1))}(\beta^{(j)})\ne\emptyset\}| \\
&=n-\dist((K-E)^{\gamma-\beta}\setminus(K-E)^{\gamma-\alpha}) \\
&=\dist(D_S^\alpha\setminus D_S^\beta)-\dist((K-E)^{\gamma-\beta}\setminus(K-E)^{\gamma-\alpha}).\qedhere
\end{aligned} 
\end{gather}
\end{proof}


\begin{prp}\label{08}
Let $S$ be a good semigroup, $E\in \GSI_S$, and $\alpha,\beta\in D_S$ with $\alpha\le\beta$.
Let $K$ be the canonical ideal of $S$.
Then the following are equivalent:
\begin{enumerate}[(i)]
\item $\dist(E^\alpha\backslash E^\beta)=\dist(D_S^\alpha\setminus D_S^\beta)-\dist((K-E)^{\gamma-\beta}\setminus(K-E)^{\gamma-\alpha})$.
\item For all $\delta\in D_S$ such that $\alpha\le\delta\le \beta$ and for every $i\in\{1,\dots,s\}$ such that $\delta+\ee_i\le\beta$,  
\[
\ol\Delta^E_i(\delta)\ne\emptyset \iff \Delta_i^{K-E}(\tau-\delta)=\emptyset.
\]
\item For all $\delta\in D_S$ such that $\alpha\le\delta\le \beta$ and for every $i\in\{1,\dots,s\}$ such that $\delta-\ee_i\ge\alpha$,
\[
\ol\Delta_i^{K-E}(\tau-\delta)\ne\emptyset \iff \Delta^E_i(\delta)=\emptyset.
\]
\end{enumerate}
\end{prp}

\begin{proof}
Let
\[
\alpha=\alpha^{(0)}<\alpha^{(1)}<\cdots<\alpha^{(n)}=\beta
\]
and
\[
\gamma-\beta=\beta^{(0)}<\beta^{(1)}<\cdots<\beta^{(n)}=\gamma-\alpha
\]
be as in Lemma \ref{05}. 
Let us denote again $J=\{0,\dots,n-1\}$.
Then, from the proof of Lemma \ref{05} (see \eqref{25}) we have
\[
\dist(E^\alpha\backslash E^\beta)=\dist(D_S^\alpha\setminus D_S^\beta)-\dist((K-E)^{\gamma-\beta}\setminus(K-E)^{\gamma-\alpha})
\]
if and only if
\[
|\{j \in J\mid \ol\Delta^E_{i(j)}(\alpha^{(j)})\ne\emptyset\}
=|\{j\in J\mid \Delta^{K-E}_{i(j)}(\tau-\alpha^{(j)})=\emptyset\}.
\]
Since the first set is contained in the second by Lemma \ref{04}, we obtain
\[
\{j \in J\mid \ol\Delta^E_{i(j)}(\alpha^{(j)})\ne\emptyset\}
=\{j\in J\mid \Delta^{K-E}_{i(j)}(\tau-\alpha^{(j)})=\emptyset\}
\]
In particular
\[
\ol\Delta^{E}_{i(j)}(\alpha^{(j)})\ne\emptyset\Longleftrightarrow \Delta^{K-E}_{i(j)}(\tau-\alpha^{(j)})=\emptyset.
\]
Now let $\delta\in D_S$ be such that $\alpha\le\delta\le \beta$ and for every $i\in\{1,\dots,s\}$, $\delta+\ee_i\le\beta$. 
Then it is always possible to find a saturated chain in $D_S$ between $\alpha$ and $\beta$ such that $\delta=\alpha^{(j)}$ and $i=i(j)$.
Thus
\[
\ol\Delta^{E}_{i}(\delta)\ne\emptyset\Longleftrightarrow \Delta^{K-E}_{i}(\tau-\delta)=\emptyset.
\]
Finally, observing that $E=K-(K-E)$ by Proposition \ref{22}, this is also equivalent to
\[
\ol\Delta^{K-E}_{i}(\tau-\delta)\ne\emptyset\Longleftrightarrow \Delta^{E}_{i}(\delta)=\emptyset.
\]
if $\delta-\ee_i\ge \alpha$ (i.e.\ $(\tau-\delta)+\ee_i\le \tau-\alpha$).
\end{proof}


The next corollary gives the necessary equivalent conditions for the main Theorem \ref{31}.

\begin{cor}\label{26}
Let $S$ be a good semigroup, $E\in \GSI_S$, and $\alpha\in D_S$ with $\mu^E\le\alpha\le\gamma^E$.
Let $K$ be the canonical ideal of $S$.
Then the following are equivalent:
\begin{enumerate}[(i)]
\item\label{26a} $\dist(D_S^{\mu^E}\backslash E)=\dist((K-E)\backslash D_S^{\gamma-\mu^E})$.
\item\label{26b} $\dist(E\backslash E^{\gamma^E})=\dist(D_S^{\mu^E}\backslash D_S^{\gamma^E})-\dist((K-E)\backslash (K-E)^{\gamma-\mu^E})$.
\item\label{26c} For every $i\in\{1,\dots,s\}$ such that $\alpha+\ee_i\le\gamma^E$,  
\[
\ol\Delta^E_i(\alpha)\ne\emptyset \iff \Delta_i^{K-E}(\tau-\alpha)=\emptyset.
\]
\item\label{26d} For every $i\in\{1,\dots,s\}$ such that $\alpha-\ee_i\ge\mu^E$,
\[
\ol\Delta_i^{K-E}(\tau-\alpha)\ne\emptyset \iff \Delta^E_i(\alpha)=\emptyset.
\]
\end{enumerate}
\end{cor}

\begin{proof}
First of all observe that by additivity (see Lemma \ref{07})
\[
\dist(D_S^{\mu^E}\backslash E)=\dist(D_S^{\mu^E}\backslash D_S^{\gamma^E})-\dist(E\backslash D_S^{\gamma^E}).
\]
As $D_S^{\gamma-\mu^E}=(K-E)^{\gamma-\mu^E}$ and $E^{\gamma^E}=D_S^{\gamma^E}$, $(i)$ is equivalent to $(ii)$.
Now observe that by Lemma \ref{14}.\eqref{14c} and Remark \ref{24}, $(ii)$ is the same as
\[
\dist(E^{\mu^E}\backslash E^{\gamma^E})=\dist(D_S^{\mu^E}\backslash D_S^{\gamma^E})-\dist((K-E)^{\gamma-\gamma^E}\backslash (K-E)^{\gamma-\mu^E}).
\]
The claim follows then trivially from Proposition \ref{08} with $\alpha=\mu^E$ and $\beta=\gamma^E$.
\end{proof}

\begin{rmk}\label{29}
Let $R$ be an admissible ring and $\E$ a fractional ideal of $R$.
Set $S=\Gamma_R$ and $E=\Gamma_\E$.
Then Remark \ref{24} and Proposition \ref{74} imply Corollary \ref{26}.\ref{26a}.
\end{rmk}

\section{Symmetry of the Poincar\'e series}

We now come to the main results of this paper.
Let us first define the main objects of study, i.e.\ the Poincar\'e series.

\begin{ntn}
For every $J\subseteq\{1,\dots,s\}$, we denote $\ee_J=\sum_{j\in J}\ee_j$.
\end{ntn}

The following definition was given in \cite[$\S$ 5.2.8]{Pol16}:


\begin{dfn}
Let $R$ be an admissible ring, and let $\E$ be a fractional ideal of $R$. 
We define
\[
\ell_\E(\alpha):=\ell(\E^\alpha/\E^{\alpha+\one}), \quad \LL_\E(\ti):=\sum_{\alpha\in D_S}\ell_\E(\alpha)\ti^\alpha,
\] 
where $\ti=(t_1,\dots,t_s)$, and $\ti^\alpha=t_1^{\alpha_1}\cdots t_s^{\alpha_s}$.

The \emph{Poincar\'e series of $\E$} is
\[
\PP_\E(\ti):=\LL_\E(\ti)\prod_{i=1}^s(t_i-1).
\]
\end{dfn}


We give an analogous definition for good semigroup ideals:

\begin{dfn}\label{09}
Let $S$ be a good semigroup, and let $E\in\GSI_S$. 
We define
\[
\dist_E(\alpha):=\dist(E^\alpha\backslash E^{\alpha+\one}), \quad \LL_E(\ti):=\sum_{\alpha\in D_S}\dist_E(\alpha)\ti^\alpha.
\] 
The \emph{Poincar\'e series of $E$} is
\[
\PP_E(\ti):=\LL_E(\ti)\prod_{i=1}^s(t_i-1).
\]
\end{dfn}


\begin{rmk}\label{32}
Let $R$ be an admissible ring, and let $\E$ be a fractional ideal of $R$.
Then Lemma \ref{15} and Proposition \ref{17} yield $\LL_{\Gamma_\E}(\ti)=\LL_\E(\ti)$, and in particular $\PP_{\Gamma_\E}(\ti)=\PP_\E(\ti)$.
\end{rmk}


The Poincar\'e series can be written in a more compact fashion.

\begin{lem}\label{10}
Let $S$ be a good semigroup, and let $E\in\GSI_S$. 
We define
\[
c_E(\alpha):=\sum_{J\subseteq\{1,\dots,s\}}(-1)^{|J^c|}\dist_E(\alpha-\ee_J)
\]
where $J^c$ denotes the complement of $J$ in $\{1,\dots,s\}$.
Then the Poincar\'e series can be written as
\[
\PP_E(\ti)=\sum_{\alpha\in D_S}c_E(\alpha)\ti^\alpha.
\]
\end{lem}

\begin{proof}
Observe that
\begin{align*}
\prod_{i=1}^s(t_i{-}1)=& t_1\cdots t_s{+}(-1)^1\sum_{\mathclap{ i_1<\cdots<i_{s-1}}}t_{i_1}\cdots t_{i_{s-1}}{+}\cdots{+}(-1)^{s-1}\sum_{i=1}^st_i{+}(-1)^s \\
=&\sum_{J\subseteq\{1,\dots,s\}}(-1)^{|J^c|}\ti^{\ee_J}.
\end{align*}
Hence
\begin{align*}
\PP_E(\ti)=&\sum_{\alpha\in D_S}\dist_E(\alpha)\ti^\alpha\prod_{i=1}^s(t_i-1)
=\sum_{\alpha\in D_S}\dist_E(\alpha)\ti^\alpha\sum_{\mathclap{J\subseteq\{1,\dots,s\}}}(-1)^{|J^c|}\ti^{\ee_J}\\
=&\sum_{\alpha\in D_S}\sum_{J\subseteq\{1,\dots,s\}}(-1)^{|J^c|} \dist_E(\alpha)\ti^{\alpha+\ee_J}= \\
=&\sum_{\alpha\in D_S}\sum_{J\subseteq\{1,\dots,s\}}(-1)^{|J^c|}\dist_E(\alpha-\ee_J)\ti^{\alpha}=\sum_{\alpha\in D_S}c_E(\alpha)\ti^\alpha.\qedhere
\end{align*}
\end{proof}


The next lemma is necessary to prove Proposition \ref{11}.

\begin{lem}\label{13}
Let $S$ be a good semigroup, $E\in\GSI_S$, and $\beta\in D_S$. 
If $\beta_i+1<\mu_i^E$ or $\beta_i>\gamma^E_i$, then $\dist_E(\beta)=\dist_E(\beta+\ee_i)$.
\end{lem}

\begin{proof}
Let $\beta=\beta^{(0)}<\beta^{(1)}=\beta+\ee_i<\cdots<\beta^{(s)}=\beta+\one<\beta^{(s+1)}=\beta+\ee_i+\one$ be a saturated chain in $D_S$, where $\beta^{(j+1)}=\beta^{(j)}+\ee_j$ for all $j\in \{1,\dots,s\}\setminus\{i\}$.
Then by definition of $\dist_E(\beta)$ and by additivity of the distance (see Lemma \ref{07}) we have
\[
\dist_E(\beta)=\dist_E(E^\beta\backslash E^{\beta+\one})=\sum_{j=0}^{s-1} \dist_E(E^{\beta^{(j)}}\backslash E^{\beta^{(j+1)}}).
\]
On the other hand we have
\[
\dist_E(\beta+\ee_i)=\dist_E(E^{\beta+\ee_i}\backslash E^{\beta+\ee_i+\one})=\sum_{j=1}^{s} \dist_E(E^{\beta^{(j)}}\backslash E^{\beta^{(j+1)}}).
\]
Therefore
\begin{align*}
\dist_E(\beta+\ee_i)-\dist_E(\beta)=& \dist_E(E^{\beta^{(s)}}\backslash E^{\beta^{(s+1)}})-\dist_E(E^{\beta^{(0)}}\backslash E^{\beta^{(1)}}) \\
=& \dist_E(E^{\beta+\one}\backslash E^{\beta+\ee_i+\one})-\dist_E(E^{\beta}\backslash E^{\beta+\ee_i}).
\end{align*}
By Lemma \ref{02} we know that 
\[
\dist_E(E^{\beta}\backslash E^{\beta+\ee_i})=1 \iff \ol\Delta^E_i(\beta)\ne\emptyset.
\]
and 
\[
\dist_E(E^{\beta+\one}\backslash E^{\beta+\ee_i+\one})=1 \iff \ol\Delta^E_i(\beta+\one)\ne\emptyset.
\]
If $\beta_i+1<\mu^E_i$, then also $\beta_i<\mu^E_i$, and therefore $\ol\Delta^E_i(\beta)=\ol\Delta^E_i(\beta+\one)=\emptyset$. This yields $\dist_E(\beta+\ee_i)-\dist_E(\beta)=0$.
On the other hand, when $\beta_i>\gamma^E_i$, then also $\beta_i+1>\gamma^E_i$ and  $\ol\Delta^E_i(\beta)\ne \emptyset$, $\ol\Delta^E_i(\beta+\one)\ne\emptyset$. This implies $\dist_E(E^{\beta}\backslash E^{\beta+\ee_i})=\dist_E(E^{\beta+\one}\backslash E^{\beta+\ee_i+\one})=1$, and thus once again $\dist_E(\beta+\ee_i)-\dist_E(\beta)=0$.
\end{proof}


We can now prove that the Poincar\'e series of a good semigroup ideal is in fact a polynomial.

\begin{prp}\label{11}
Let $S$ be a good semigroup, and let $E\in\GSI_S$.
Then $\PP_E(\ti)$ is a polynomial.
\end{prp}

\begin{proof}
The goal is to prove that $c_E(\alpha)\ne 0$ only if $\mu^E\le \alpha\le\gamma^E$. 
Suppose there exists an $i$ such that $\alpha_i<\mu_i^E$. 
Consider $J\subseteq\{1,\dots,s\}$. It is not restrictive to consider $i\not\in J$ (otherwise we can consider $J\setminus\{i\}$). 
Notice that if $\alpha-\ee_{J\cup\{i\}}=\beta$, then $\alpha-\ee_J=\beta+\ee_i$. Since $\alpha_i<\mu_i^E$, then $\mu^E_i>(\alpha-\ee_J)_i=(\beta+\ee_i)_i=\beta_i+1$. 
So by Lemma \ref{13}, we have
\[
\dist_E(\alpha-\ee_{J\cup\{i\}})=\dist_E(\alpha-\ee_J).
\]
The same is true similarly if $i$ is such that $\alpha_i>\gamma_i^E$. 
Therefore when $\alpha\not\in\{\beta\mid \mu^E\le\beta\le\gamma^E\}$, for each $J\subseteq \{1,\dots,s\}$ there exists a $J'\subset\{1,\dots,s\}$ (either $J\cup\{i\}$ or $J\setminus\{i\}$) such that
\[
\dist_E(\alpha-\ee_{J'})=\dist_E(\alpha-\ee_J)
\]
and $|J|=|J'|\pm 1$. 
Hence these terms annihilate each other in the sum 
\[
\sum_{J\subseteq \{1,\dots,s\}}(-1)^{|J^c|}\dist_E(\alpha-\ee_J),
\]
so that $c_E(\alpha)=0$ for all $\alpha\not\in\{\zeta\mid \mu^E\le\zeta\le\gamma^E\}$.

Thus $\PP_E(\ti)$ is a polynomial.
\end{proof}


Finally, we are ready to prove our main theorem.

\begin{thm}\label{31}
Let $S$ be a good semigroup, and let $E\in\GSI_S$.
Let $K$ be the canonical ideal of $S$.
If one of the equivalent conditions of Corollary \ref{26} holds, then the Poincar\'e polynomials of $E$ and $K-E$ are symmetric:
\[
\PP_{K-E}(\ti)=(-1)^{s+1}\ti^\gamma\PP_E\left(\frac{1}{\ti}\right).
\]
\end{thm}

\begin{proof}
By Lemma \ref{10}, $\PP_{K-E}(\ti)=\sum_{\alpha\in D_S}c_{K-E}(\alpha)\ti^\alpha$, while
\begin{align*}
(-1)^{s+1}\ti^\gamma \PP_E\left(\frac{1}{\ti}\right)=&(-1)^{s+1}\ti^\gamma\sum_{\beta\in D_S}c_{E}(\beta)\ti^{-\beta}\\
=&\sum_{\beta\in D_S}(-1)^{s+1}c_{E}(\beta)\ti^{\gamma-\beta}\\
=&\sum_{\alpha\in D_S}(-1)^{s+1}c_{E}(\gamma-\alpha)\ti^\alpha.
\end{align*}
Therefore the claim is equivalent to
\[
c_{K-E}(\alpha)=(-1)^{s+1}c_{E}(\gamma-\alpha).
\]
If $\alpha\not\in\{\zeta\mid \mu^E\le\gamma-\zeta\le\gamma^E\}=\{\zeta\mid \gamma-\gamma^E\le\zeta\le\gamma-\mu^E\}$ then $c_{K-E}(\alpha)=c_E(\gamma-\alpha)=0$ by proof of Proposition \ref{11}. 
So we can assume $\gamma-\gamma^E\le\alpha\le\gamma-\mu^E$.

Now let $\alpha=\gamma-\beta$. 
Then $\mu^E\le\beta\le\gamma^E$.
As the equivalent conditions of Corollary \ref{26} are satisfied, for any $\delta$ such that $\mu^E\le\delta\le\gamma^E$ with $\delta+\ee_i
\le\gamma^E$, $\ol\Delta^E_i(\delta)\ne\emptyset$ if and only if $\Delta^{K-E}_i(\tau-\delta)=\emptyset$.
In particular, for any $\delta$ with $\mu^E\le\beta-\one\le\delta\le\beta\le\gamma^E$, $\ol\Delta^E_i(\delta)\ne\emptyset$ if and only if $\Delta^{K-E}_i(\tau-\delta)=\emptyset$.
Hence by Proposition \ref{08}, $\dist(E^{\beta-\one}\backslash E^{\beta})=\dist(D_S^{\beta-\one}\backslash D_S^{\beta})-\dist((K-E)^{\gamma-\beta}\backslash (K-E)^{\gamma-\beta+\one})$.
Now recalling that $\alpha=\gamma-\beta$ we have $\dist(E^{\gamma-\alpha-\one}\backslash E^{\gamma-\alpha})=\dist(D_S^{\gamma-\alpha-\one}\backslash D_S^{\gamma-\alpha})-\dist((K-E)^{\alpha}\backslash (K-E)^{\alpha+\one})$.
As $\dist(D_S^{\gamma-\alpha-\one}\backslash D_S^{\gamma-\alpha})=d_{D_S}(\gamma-\alpha-\one,\gamma-\alpha)=s$, this translates to
\[
\dist_{K-E}(\alpha)=s-\dist_E(\gamma-\alpha-\one),
\]
for any $\gamma-\gamma^E\le\alpha\le\gamma-\mu^E$ with $\alpha+\one\le\gamma-\mu^E$.
Then
\begin{align*}
c_{K-E}&(\alpha)= \sum_{J\subseteq \{1,\dots,s\}}(-1)^{|J^c|}\dist_{K-E}(\alpha-\ee_J)\\
=&(-1)^s\sum_{J\subseteq \{1,\dots,s\}}(-1)^{|J|}(s-\dist_E(\gamma{-}\alpha{-}\one{+}\ee_J))\\
=& (-1)^ss\sum_{\mathclap{J\subseteq \{1,\dots,s\}}}(-1)^{|J|}+(-1)^{s+1}\sum_{\mathclap{J\subseteq \{1,\dots,s\}}}(-1)^{|J|}\dist_E(\gamma{-}\alpha{-}\one{+}\ee_J)\\
=&(-1)^ss\sum_{i=0}^s(-1)^i\binom{s}{i}+(-1)^{s+1}\sum_{\mathclap{J\subseteq \{1,\dots,s\}}}(-1)^{s+|J^c|}\dist_E(\gamma{-}\alpha{-}\ee_{J^c}) \\
=&(-1)^s(1-1)^s+(-1)^{s+1}c_E(\gamma-\alpha) \\
=&(-1)^{s+1}c_E(\gamma-\alpha).
\end{align*}
Hence the claim.
\end{proof}


As a corollary, we obtain a generalization of Pol's result \cite[Proposition 5.2.28]{Pol16}.

\begin{cor}\label{27}
Let $R$ be an admissible ring, $\E$ a fractional ideal of $R$ and $\K$ a canonical ideal of $R$ such that $R\subseteq \K\subseteq \ol R$.
Set $E=\Gamma_\E$ and $K=\Gamma_\K$. 
Then:
\[
\PP_{K-E}(\ti)=(-1)^{s+1}\ti^\gamma \PP_E\left(\frac{1}{\ti}\right).
\]
\end{cor}

\begin{proof}
It follows immediately from Remarks \ref{29} and \ref{32}, and Theorem \ref{31}.
\end{proof}


\begin{rmk}
Remark \ref{29} shows that the equivalent conditions of Corollary \ref{26} are true in 
the value semigroup case.
It remains the question whether they are always satisfied.
If not, they could represent a step forward in characterizing the class of value semigroups inside the bigger class of good semigroups.
\end{rmk}

\bibliographystyle{amsalpha}
\bibliography{Poincare_series}

\end{document}